\documentclass[11pt,reqno]{amsart}
\usepackage{amssymb,dsfont,enumerate}
\usepackage[pagebackref]{hyperref}
\usepackage[portrait,a4paper,margin=3cm]{geometry}

\newtheorem{thm}{Theorem}%
\newtheorem{cor}[thm]{Corollary}%
\newtheorem{lem}[thm]{Lemma}%

\newcommand{\dC}{\mathbb{C}}
\newcommand{\dE}{\mathbb{E}}

\newcommand{\dR}{\mathbb{R}}

\newcommand{\cD}{\mathcal{D}}

\newcommand{\ABS}[1]{{{\left| #1 \right|}}} % |1|
\renewcommand{\leq}{\leqslant}             % redef. of < or =
\renewcommand{\geq}{\geqslant}             % redef. of > or =
\newcommand{\wt}{\widetilde}
\newcommand{\ind}{\mathds{1}}

\newcommand{\tr}{\mathrm{tr}}

\renewcommand{\Im}{\mathfrak{Im}}

\newcommand{\weak}{\rightsquigarrow}
\title[Spectrum of non-hermitian random matrices]{On the spectrum of sum and product of non-hermitian random matrices}

\date{Preprint, compiled \today}

\author{Charles Bordenave}
\address[Ch.~Bordenave]{IMT UMR 5219 CNRS and Universit\'e Paul-Sabatier Toulouse III, France}
\email{charles.bordenave(at)math.univ-toulouse.fr}
\urladdr{http://www.math.univ-toulouse.fr/~bordenave/}



\begin{document}

\maketitle

\begin{abstract}
In this short note, we revisit the work of T. Tao and V. Vu on large non-hermitian random matrices with independent and identically distributed (i.i.d.) entries with mean zero and unit variance. We prove under weaker assumptions that the limit spectral distribution of sum and product of non-hermitian random matrices is universal. As a byproduct, we show that the generalized eigenvalues distribution of two independent matrices converges almost surely to the uniform measure on the Riemann sphere.  \end{abstract}

\section{Introduction}

We start with some usual definitions. We endow the space of probability measures on $\dC$ with the topology of weak convergence: a sequence of probability measures $(\mu_n)_{n \geq 1}$ \emph{converges weakly} to $\mu$ is for any bounded continuous function $f : \dC \to \dR$,
$$
 \int f d\mu_n - \int f d\mu 
$$
converges to $0$ as $n$ goes to infinity. In this note, we shall denote this convergence by $\mu_n \underset{n\to\infty}{\weak} \mu$. Similarly, for two sequences of probability measures $(\mu_n)_{n \geq 1}, (\mu'_n)_{n \geq 1}$, we  will use $\mu_n - \mu'_n \underset{n\to\infty}{\weak} 0$, or say that $\mu_n - \mu_n'$ tends weakly to $0$, if 
$$
 \int f d\mu_n - \int f d\mu'_n $$ converges to $0$ for any bounded continuous function $f$. We will say that a measurable function $f : \dC \to \dR$ is \emph{uniformly bounded} for $(\mu_n)_{n \geq 1}$ if  $$
\limsup_{n \to \infty}  \int | f | d\mu_n < \infty.
 $$
Finally, recall  that a function $f$ is \emph{uniformly integrable} for $(\mu_n)_{n \geq 1}$ if 
$$
\lim_{ t \to + \infty} \limsup_{n  \to \infty}  \int_{|f |Ê\geq t}   | f | d\mu_n = 0.
$$ 
The above definitions will also be used for probability measures on $\dR_+ = [0, \infty)$ and functions $f : \dR_+ \to \dR$.

The \emph{eigenvalues} of an $n\times n$ complex matrix $M$ are the roots in
$\dC$ of its characteristic polynomial. We label them
$\lambda_1(M),\ldots,\lambda_n(M)$ so that
$|\lambda_1(M)|\geq\cdots\geq|\lambda_n(M)|\geq0$. We also denote by
$s_1(M)\geq\cdots\geq s_n(M) \geq 0$ the \emph{singular values} of $M$, defined for
every $1\leq k\leq n$ by $s_k(M):=\lambda_k(\sqrt{MM^*})$ where
$M^*$ is the conjugate transpose of $M$. We define the
empirical spectral measure and the empirical singular values measure as
$$
\mu_M = \frac 1 n \sum_{k=1} ^n \delta_{\lambda_k (M)} %
\quad \text{and } \quad %
\nu_M = \frac 1 n \sum_{k=1} ^n \delta_{s_k (M)}.
$$
Note that $\mu_M$ is a probability measure on $\dC$ while $\nu_M$ is a probability measure on $\dR_+$. The \emph{generalized eigenvalues} of $(M,N)$, two $n\times n$ complex matrices, are the zeros of the polynomial $\det ( M - z N)$. If $N$ is invertible, it is simply the eigenvalues of $N^{-1} M$. 

Let $(X _{ij})_{ i , j \geq 1}$ and $(Y _{ij})_{ i , j \geq 1}$ be independent i.i.d. complex random variables with mean $0$ and variance $1$. 
Similarly, let $(G _{ij})_{ i , j \geq 1}$ and $(H _{ij})_{ i , j \geq 1}$ be independent complex
centered gaussian variables with variance $1$, independent of $(X_{ij},Y_{ij})$.  We consider the random matrices $X_n = (X_{ij}) _{ 1  \leq i, j \leq n}$, $Y_n = (Y_{ij}) _{ 1  \leq i, j \leq n}$, $ G_n  = (G_{ij}) _{ 1  \leq i, j \leq n}$ and  $H_n  = (H_{ij}) _{ 1  \leq i, j \leq n}$. For ease of notation, we will sometimes drop the subscript $n$. It is known that almost surely (a.s.) for $n$ large enough, $X$ is invertible (see the forthcoming Theorem \ref{le:small}) and then $\mu_{ X^{-1} Y}$ is a well defined random probability measure on $\dC$.

Now, let $\mu$ be the probability measure whose density with respect to the Lebesgue measure on $\dC \simeq \dR^2$ is 
$$
\frac{1} { \pi ( 1 + |z|^2 )^2}. 
$$
Through stereographic projection, $\mu$ is easily seen to be the uniform measure on the Riemann sphere. Haagerup and several authors afterwards have independently observed the following beautiful identity (see Krishnapur \cite{K}, Rogers \cite{R} and Forrester and Mays \cite{FM}). 

\begin{lem}[Spherical ensemble]\label{prop:riemann} 
For each integer $n \geq 1$, $$
\dE \mu_{ G^{-1} H}   =  \mu. 
$$
\end{lem}

By reorganizing the results of Tao and Vu \cite{MR2409368,tao-vu-cirlaw-bis}, we will prove a universality result. 

\begin{thm}[Universality of generalized eigenvalues]\label{prop:taovuuniversality}
Almost surely, $$
\mu_{ÊX^{-1} Y } - \mu_{ G^{-1} H}   \underset{n\to\infty}{\weak}  0.$$
\end{thm}

Applying once Lemma \ref{prop:riemann} and twice Theorem \ref{prop:taovuuniversality}, we get
\begin{cor}[Spherical law]\label{cor:riemann}
Almost surely
 $$
\mu_{ÊX^{-1} Y } \underset{n\to\infty}{\weak}  \mu. $$
\end{cor}

This statement was recently conjectured in \cite{R,FM}. More generally, our argument also leads to the following universality result for sums and products of random matrices. 
\begin{thm}[Universality of sum and product of random matrices]\label{th:univ}
For every integer $n$, let $M_n,K_n,L_n$ be $n \times n$ complex matrices such that, for some $\alpha > 0$, 
 \begin{itemize}
    \item[(i)] $ x \mapsto  x^{-\alpha}$ is uniformly bounded for
     $(\nu_{K_n})_{n\geq1}$, and  $(\nu_{L_n})_{n\geq1}$ and $ x \mapsto x^{\alpha}$ is uniformly bounded for $(\nu_{M_n})_{n\geq1}$,
    
    \item[(ii)] for almost all (a.a.) $z \in \dC$, $\nu_{K_n^{-1} M_n L_n^{-1}  - K_n^{-1} L_n^{-1}  z }$ converges weakly to a probability measure $\nu_z$.
    \end{itemize}
Then, almost surely, 
$$
\mu_{ÊM + K X L / \sqrt n}   \underset{n\to\infty}{\weak}  \mu, 
$$
where $\mu$ depends only $(\nu_z)_{z \in \dC}$. 
\end{thm}

For $M = K = L = I$, the identity matrix, this statement gives the famous circular law theorem, that was established through a long sequence of partial results \cite{MR0220494,MR773436,MR2130247,MR841088,MR1437734,MR1454343,MR1428519,MR2191234,bai-silverstein-book,1687963,gotze-tikhomirov-new,MR2409368,tao-vu-cirlaw-bis}. In this note, the steps of proof are elementary and they borrow all difficult technical statements from previously known results. Nevertheless, this theorem generalizes Theorem 1.18 in  \cite{tao-vu-cirlaw-bis} in two directions.  First, we have removed the assumption of  uniformly bounded second moment for $\nu_{ M + K X L / \sqrt n }$, $\nu_{K^{-1} M L^{-1}}$ and  $\nu_{K^{-1} L^{-1}}$. Secondly, it proves the convergence of the spectral measure. The explicit form of $\mu$ in terms of $\nu_z$ is quite complicated. It is given by the forthcoming equations (\ref{eq:defmu}-\ref{eq:resDS}). This expression is not very easy to handle. However, following ideas developed in \cite{R} or using tools from free probability as in \cite{MR1929504,GMZ},  it should be possible to find more elegant formulas. For nice examples of limit spectral distributions, see e.g. \cite{R}. It is interesting to notice that we may deal with non-centered variables $(X_{ij})$ by including the average matrix of $X/ \sqrt n $ into $M$, and recover \cite{djalil-nccl}. Finally, as in \cite{GT}, it is also possible by induction to apply  Theorem \ref{th:univ} to product  of independent copies of $X$ (with the use of the forthcoming Theorem \ref{th:UI}).

\section{Proof of Theorem \ref{prop:taovuuniversality}}

\subsection{Replacement Principle}

The following is an extension of Theorem 2.1 in \cite{tao-vu-cirlaw-bis}.  The idea goes back essentially to Girko.

\begin{lem}[Replacement principle]\label{le:girko}
  Let $A_n, B_n$ be $n \times n$ complex random matrices. Suppose that for a.a.\ $z\in\dC$, a.s.\
  \begin{itemize}
  \item[(i)] $\nu_{A_n-z} - \nu_{B_n-z}$ tends weakly to $0$, 
  \item[(ii)] $\ln(\cdot)$ is uniformly integrable for
    $(\nu_{A_n-z})_{n\geq1}$ and  $(\nu_{B_n-z})_{n\geq1}$.
  \end{itemize}
 Then a.s. $\mu_{A_n} - \mu_{B_n}$ tends weakly to $0$. Moreover the same holds if we replace (i) by
    \begin{itemize}
  \item[(i')] $\int \ln (x) d \nu_{A_n-z} - \int \ln (x) d  \nu_{B_n-z}$ tends to $0$. 
    \end{itemize}
  \end{lem}

\begin{proof}
It is a straightforward adaptation of \cite[Lemma A.2]{cirmar}. 
\end{proof}

\begin{cor}\label{cor:girko}
Let $A_n, B_n , M_n$ be $n \times n$ complex random matrices. Suppose that a.s. $M_n$ is invertible and for a.a.\ $z\in\dC$, a.s.\
   \begin{itemize}
  \item[(i)] $\nu_{A_n-z M_n ^{-1}} - \nu_{B_n-z M_n^{-1}}$ tends weakly to $0$, 
  \item[(ii)] $\ln(\cdot)$ is uniformly integrable for
    $(\nu_{A_n-z M_n^{-1} })_{n\geq1}$ and  $(\nu_{B_n-z M_n^{-1} })_{n\geq1}$.
    \end{itemize}
 Then a.s. $\mu_{M_n A_n} - \mu_{M_n B_n}$ tends weakly to $0$. 
  \end{cor}
  \begin{proof}
If $M_n$ is invertible, note that 
$$\int \ln (x) d \nu_{ M_n A_n -z } = \frac 1 n \ln Ê|Ê\det  ( M_n A_n - z )  |  = \int \ln (x) d \nu_{  A_n - z M^{-1}_n } + \frac 1 n \ln |Ê\det M_n |Ê.$$ 
We may thus apply Lemma \ref{le:girko}(i')-(ii). Indeed, in the expression $\int \ln (x) d \nu_{A_n-z} - \int \ln (x) d  \nu_{B_n-z}$, the term $\frac 1 n \ln |Ê\det M_n |Ê$ cancels. 
\end{proof}

\subsection{Convergence of singular values}

The following result is due to Dozier and Silverstein.
\begin{thm}[Convergence of singular values, \cite{DS}]\label{th:DS}
Let $(M_n)_{n \geq 1}$ be a sequence of $n \times n$ complex matrices such that $\nu_{M_n} $ converges weakly to a probability measure $\nu$. Then a.s. $\nu_{X_n / \sqrt n  + M_n}$ converges weakly to a probability measure $\rho$ which depends only on $\nu$.
\end{thm}

The measure $\rho$ has an explicit characterization in terms of $\nu$. Its exact form is not relevant here.

\subsection{Uniform integrability} In order to use the replacement principle, it is necessary to prove the uniform integrability of $\ln( \cdot  )$ for some empirical singular values measures. This is achieved by proving that,  for some $\beta >0$, $x \mapsto x^{-\beta} + x^\beta$ is uniformly bounded. 

\begin{thm}[Uniform integrability]\label{th:UI}
Let $(M_n)_{n \geq 1}$ be a sequence of $n \times n$ complex matrices, and assume that $x \mapsto x^\alpha$ is uniformly bounded for $(\nu_{M_n})_{n \geq 1}$ for some $\alpha >0$. Then there exists $\beta >0$  such that a.s. $x  \mapsto x^{-\beta} + x^\beta$ is uniformly bounded for $(\nu_{X_n / \sqrt n +  M_n})_{ n \geq 1}$. 
\end{thm}

In the remainder of the paper, the notation $n\gg1$ means \emph{large enough $n$}. We start with an elementary lemma.

 \begin{lem}[Large singular values]\label{le:large}
Almost surely, for $n \gg 1$,  
$$\int x^{2} d \nu_{X / \sqrt n} \leq 2 .$$ 
 \end{lem}
 
  \begin{proof}
We have $ \frac {1}{n} \sum_{i=1}^n s_i^2  ( X/ \sqrt n ) = \frac {1}{ n^2}  \tr X^* X = \frac {1}{ n^2} \sum_{1 \leq i,j \leq n} |X_{ij} |^2$,
and the latter converges a.s. to $1$ by the law of large number. 
 \end{proof}

 \begin{cor}
 \label{cor:large}
Let $0 < \alpha \leq 2$ and let $(M_n)_{n \geq 1}$ be a sequence of $n \times n$ complex matrices such that $x \mapsto x^{\alpha}$ is uniformly bounded for $(\nu_{M_n})_{n \geq 1}$. Then, a.s.  $x \mapsto x^{\alpha}$  is uniformly bounded for $(\nu_{X_n / \sqrt n + M_n})_{n \geq 1}$.
 \end{cor}
  
\begin{proof}
If $M, N$ are $n \times n$ complex matrices, from \cite[Theorem 3.3.16]{MR1091716}, for all $1 \leq i  ,  j \leq n$ with $1 \leq i + j \leq n + 1$,  
\begin{equation*}\label{eq:weyl}
s_{i + j  - 1} ( M  + N ) \leq s_{i} ( M ) + s_j ( N).
\end{equation*}
Hence,
$$
s_{2i} ( M + N ) \leq s_{2 i  - 1 }(M + N ) \leq s_{i} ( M ) + s_i ( N).
$$
We deduce that for any non-decreasing function, $f :  \dR_+ \to  \dR_+$ and $t >0$, 
\begin{equation*}\label{eq:weylunif}
\int f (x) d \nu_{ M + N} \leq  2 \int f (2x)  d \nu_{ M} +  2 \int f (2x )  d \nu_{N}, 
\end{equation*}
where we have used the inequality
\begin{equation*}
f ( x + y )  \leq f ( 2x)  + f (2y ) . 
\end{equation*}
Now, in view of Lemma \ref{le:large}, we may apply the above inequality to $f(x) = x^\alpha$ and deduce the statement. 
\end{proof}

The above corollary settles the problem of the uniform integrability of $\ln ( \cdot)$ at $+ \infty$ for $\nu_{X/ \sqrt n + M}$. The uniform integrability at $0+$ is a much more delicate matter.  The next theorem is a deep result of Tao and Vu. 
 \begin{thm}[Small singular values, \cite{MR2409368,tao-vu-cirlaw-bis}]\label{le:small}
Let $(M_n)_{n \geq 1}$ be a sequence of $n \times n$ complex invertible matrices such that $x \mapsto x^\alpha$ is uniformly bounded for $(\nu_{M_n})_{n \geq 1}$ for some $\alpha >0$.  There exist  $c_1, c_0>0$ such that  a.s. for $n \gg 1$, 
 $$
s_{n} ( X_n / \sqrt n  + M_n  ) \geq n^{-c_1}.
 $$
Moreover for $i \geq n^{1 - \gamma}$ with $\gamma = 0.01$, a.s. for $n \gg 1$, 
 $$
 s_{n - i } ( X_n / \sqrt n  + M_n ) \geq  c_0 \frac{i}{n}.
 $$
\end{thm}

 \begin{proof}
The first statement is  Theorem 2.1 in \cite{MR2409368} and the second is contained in \cite{tao-vu-cirlaw-bis} (see the proof of Theorem 1.20 and observe that the statement of Proposition 5.1 remains unchanged if we consider a row of the matrix $X_n  + \sqrt n M_n$).
 \end{proof}

\begin{proof}[Proof of Theorem \ref{th:UI}]

By Corollary \ref{cor:large}, it is sufficient to prove that $x \mapsto x^{-\beta}$ is uniformly bounded  for $(\nu_{X/ \sqrt n + M})$ and some $\beta >0 $. We have
 \begin{equation*}\label{eq:tightnessa}
 \limsup_n \frac{1}{n} \sum_{i =1 }^n  s^{-\beta}_{i} (X/ \sqrt n  + M  ) < \infty, 
 \end{equation*}
By Theorem \ref{le:small}, we may a.s. write for $n \gg 1$, 
\begin{eqnarray*}
\frac{1}{n} \sum_{i =1 }^n  s^{-\beta}_{i} (X/ \sqrt n  + M  ) & \leq & \frac{1}{n} \sum_{i =1 }^{\lfloor n^{1-\gamma} \rfloor } n^{\beta c_1  } + \frac{1}{n}\sum_{i = \lfloor n^{1-\gamma}  \rfloor +1}^n c_2  \left( \frac{n}{i} \right)^{\beta} \\
& \leq & n^{\beta c_1  - \gamma } + \frac{1}{n} \sum_{i =1}^n c_2  \left( \frac{n}{i} \right)^{\beta}.
\end{eqnarray*}
 This last expression is uniformly bounded if $0 < \beta < \min ( \gamma /  c_1  , 1)$. \end{proof}

\subsection{End of proof of Theorem \ref{prop:taovuuniversality}}
  
If $\rho$ is a probability measure on $\dC\backslash \{0\}$, we define $\check  \rho$ as the pull-back measure of $\rho$ under $\phi : z \mapsto 1 / z$,  for any Borel $E$ in  $\dC\backslash \{0\}$, $\check  \rho (E) = \rho (\phi^{-1} (E))$. Obviously, if $(\rho_n)_{n \geq 1}$ is a sequence of probability measures on $\dC\backslash \{0\}$, then $\rho_n$ converges weakly to $\rho$ is equivalent to $\check \rho_n$ converges weakly to $\check \rho$.

Note that by Theorem \ref{th:UI}, a.s. for $n \gg 1$, $X_n$ is invertible and $x \mapsto x^{-\beta} +  x^{\beta}$ is uniformly bounded for $(\nu_{X_n/\sqrt n})_{n \geq 1}$. Also, from the quarter circular law theorem, $\nu_{X_n/ \sqrt n}$ converges a.s. to a probability distribution with density 
\begin{equation*}
\frac{1}{\pi}\sqrt{4-x^2}\ind_{[0,2]}(x),
\end{equation*}
(see Marchenko-Pastur theorem \cite{marchenko-pastur,MR0467894,MR862241}). From the independence of $(X_{ij}),(Y_{ij}),(G_{ij}),(H_{ij})$, we may apply Corollary \ref{cor:girko}, Theorem \ref{th:DS} and Theorem \ref{th:UI} conditioned on $(X_{ij})$ to $M_n = z X_n / \sqrt n $. We get a.s. 
$$
\mu_{X^{-1} Y} - \mu_{X^{-1} H} \underset{n\to\infty}{\weak}  0. 
$$
By Theorem \ref{le:small}, a.s. for $n \gg 1$, $X^{-1} H$ and $G^{-1} H$ are invertible, it follows that 
$$
\mu_{X^{-1} H} - \mu_{G^{-1} H} \underset{n\to\infty}{\weak}  0 \quad \hbox{ is equivalent to } \quad \check \mu_{X^{-1} H} - \check \mu_{G^{-1} H} \underset{n\to\infty}{\weak} 0.
$$
 However since $\mu_{MN} = \mu_{NM}$ and $\check \mu_{M} =  \mu_{M^{-1}}$, we get  
  $$
\mu_{X^{-1} H} - \mu_{G^{-1} H} \underset{n\to\infty}{\weak}  0 \quad \hbox{ is equivalent to } \quad  \mu_{H^{-1} X} -  \mu_{H^{-1} G} \underset{n\to\infty}{\weak} 0.
$$
The right hand side holds by applying again, Corollary \ref{cor:girko}, Theorem \ref{th:DS} and Theorem \ref{th:UI}.

\section{Proof of Theorem \ref{th:univ}}

\subsection{Bounds on singular values}

\begin{lem}[Singular values of sum and product]\label{le:kyfan}
If $M, N$ are $n\times n$ complex  matrices, for any $\alpha > 0$, 
\begin{eqnarray*}
 \int x^{\alpha} d \nu_{M +  N} &  \leq & 2^{1+\alpha}  \left( \int x^{\alpha} d \nu_{M }  +  \int x^{\alpha} d \nu_{N}  \right),  \\
   \int x^{\alpha} d \nu_{M N} & \leq&   2 \left( \int x^{2 \alpha} d \nu_M \right)^{1/2}\left( \int x^{2 \alpha} d \nu_N \right)^{1/2}.
\end{eqnarray*}
\end{lem}
\begin{proof}
The first statement was already treated in the proof of Corollary \ref{cor:large}. Also, from \cite[Theorem 3.3.16]{MR1091716}, for all $1 \leq i  ,  j \leq n$ with $1 \leq i + j \leq n +1 $,  
\begin{equation*}\label{eq:weyl}
s_{i + j -1 } ( M  N ) \leq s_{i} ( M ) s_{j} ( N).
\end{equation*}
Hence,
$$
s_{2i} ( M N ) \leq s_{2 i-1}(M N ) \leq s_{i} ( M ) s_i ( N).
$$
We deduce 
\begin{equation*}\label{eq:weylunif}
\int x^\alpha  d \nu_{ M  N} \leq \frac 2 n \sum_{i=1}^n s^\alpha_{i} ( M ) s^\alpha_i ( N).
\end{equation*}
We conclude by applying the Cauchy-Schwarz inequality. 
\end{proof}

\subsection{Logarithmic potential and Girko's hermitization method}

We denote by $\cD'(\dC)$ the set of Schwartz distributions endowed with its usual convergence with respect
to all infinitely differentiable functions with bounded support. Let $\mathcal{P}(\dC)$ be the set of probability measures on $\dC$ which
integrate $\ln\ABS{\cdot}$ in a neighborhood of infinity. For every
$\mu\in\mathcal{P}(\dC)$, the \emph{logarithmic potential} $U_\mu$ of $\mu$ on
$\dC$ is the function $U_\mu:\dC\to[-\infty,+\infty)$ defined for every
$z\in\dC$ by
\begin{equation*}\label{eq:logpot}
  U_\mu(z)=  \int_{\dC}\!\ln|z-z'|\,\mu(dz'),
\end{equation*}
(in classical potential theory, the definition is opposite in sign). Since
$\ln\ABS{\cdot}$ is Lebesgue locally integrable on $\dC$, one can check by
using the Fubini theorem that $U_\mu$ is Lebesgue locally integrable on $\dC$.
In particular, $U_\mu<\infty$ a.e. (Lebesgue almost everywhere) \ and
$U_\mu\in\mathcal{D}'(\dC)$. Since $\ln\ABS{\cdot}$ is the fundamental
solution of the Laplace equation in $\dC$, we have, in $\mathcal{D}'(\dC)$,
\begin{equation}\label{eq:lap}
  \Delta U_\mu= \pi \mu,
\end{equation}
where $\Delta$ is the Laplace differential operator on $\mathbb{C}$ is given by
$
\Delta =\frac{1}{4}(\partial^2_x+\partial^2_y).
$

We now state an alternative statement of Lemma \ref{le:girko} which is closer to Girko's original method, for a proof see \cite[Lemma A.2]{cirmar}.
\begin{lem}[Girko's hermitization method]\label{le:girko2}
  Let $A_n$ be a $n \times n$ complex random matrix. Suppose that for a.a.\ $z\in\dC$, a.s.\
  \begin{itemize}
  \item[(i)] $\nu_{A_n-z}$ tends weakly to a probability measure $\nu_z$ on $\dR_+$, 
  \item[(ii)] $\ln(\cdot)$ is uniformly integrable for
    $(\nu_{A_n-z})_{n\geq1}$.
  \end{itemize}
 Then there exists a probability measure $\mu\in\mathcal{P}(\dC)$ such
  that a.s. 
  \begin{itemize}
  \item[(j)]  $\mu_{A_n}$ converges weakly to $\mu$ 
  \item[(jj)] for a.a. $z\in\dC$,  
     $$
   U_\mu (z)  =  \int \!\ln(x)\,d \nu_z.
    $$
  \end{itemize}
Moreover the same holds if we replace (i) by
    \begin{itemize}
  \item[(i')] $\int \ln (x) d \nu_{A_n-z}$ tends to $\int \!\ln(x)\,d \nu_z$. 
    \end{itemize}
  \end{lem}

\begin{cor}\label{cor:girko2}
Let $A_n, K_n , M_n$ be $n \times n$ complex random matrices. Suppose that a.s. $K_n$ is invertible and $\ln(\cdot)$ is uniformly bounded for $(\nu_{K_n})_{n \geq 1}$, and for a.a.\ $z\in\dC$, a.s.\
   \begin{itemize}
  \item[(i)] $\nu_{A_n +  K_n ^{-1} (M_n -z)}$ tends weakly to a probability measure $\nu_z$,
  \item[(ii)] $\ln(\cdot)$ is uniformly integrable for
    $(\nu_{A_n+ K_n ^{-1} (M_n -z)})_{n\geq1}$.
      \end{itemize}
 Then there exists a probability measure $\mu\in\mathcal{P}(\dC)$ such
  that a.s. 
  \begin{itemize}
  \item[(j)]  $\mu_{M_n + K_n A_n}$ converges weakly to $\mu$,
  \item[(jj)] in $\mathcal{D}'(\dC)$,
     $$
  \mu   = \frac 1 \pi \Delta  \int \!\ln(x)\,d \nu_z.
    $$
  \end{itemize}
   \end{cor}
  \begin{proof}
If $K_n$ is invertible, we write
$$\int \ln (x) d \nu_{ M_n + K_n A_n -z } = \int \ln (x) d \nu_{  A_n + K_n ^{-1} (M_n -z) } + \frac 1 n \ln |Ê\det K_n |Ê.$$ 
By assumption, $ \frac 1 n \ln |Ê\det K_n |Ê=  \int \ln(x) d\nu_{K_n} $ is a.s. bounded. We may thus consider any converging subsequence and apply Lemma \ref{le:girko}(i')-(ii) together with \eqref{eq:lap}. 
\end{proof}

\subsection{End of proof of Theorem \ref{th:univ}}

We first notice that 
$$
\mu_{M + K X L / \sqrt n } = \mu_{L M L^{-1} + L K X / \sqrt n}. 
$$ 
It is thus sufficient to prove that the right hand side converges. We set $\wt M = L M L^{-1}$ and $\wt K = L K$. Since $\wt K^{-1} ( \wt M - z) = K^{-1} M L^{-1} - K^{-1} L^{-1} z$, we may apply Lemma \ref{le:kyfan} and deduce that $x \mapsto x^{\alpha/4}$ is uniformly bounded for $(\nu_{\wt K_n ( \wt M_n - z) })_{n \geq 1}$. It only remains to invoke Theorem \ref{th:UI} and Theorem \ref{th:DS} applied to $\wt K^{-1} ( \wt M - z)$, and use Corollary \ref{cor:girko2} for $\wt M + \wt KX / \sqrt n$.

\subsection{Explicit expression of the limit spectral measure}

Let $\dC_+ = \{ z \in \dC: \Im (z) >0 \}$, for a probability measure $\rho$ on $\dR$, its \emph{Cauchy-Stieltjes transform} is defined as, for all $z \in \dC_+$, 
$$
m_{\rho} (z) = \int \frac{1}{x - z} d\rho(x).
$$
By Corollary \ref{cor:girko2} and Theorem \ref{th:DS}, in $\mathcal{D}'(\dC)$,
\begin{equation}\label{eq:defmu}
\mu =  \frac {1}{2 \pi} \Delta  \int \!\ln(x)\,d \rho_z (x),
\end{equation}
where for $z \in \dC$, $\rho_z$ is a probability distribution on $\dR_+$. From \cite{DS}, for a.a. $z \in \dC$, $\rho_z$ has a Cauchy-Stieltjes transform that satisfies the integral equation: for all $w \in \dC_+$, 
\begin{equation}\label{eq:resDS}
m_{\rho_z} (w) =\int \frac{2 x (1 + m_{\rho_z} (w) ) }{x^2 - (1 + m_{\rho_z} (w) )^2  w }d \nu_z (x),
\end{equation}
where $\nu_z$ is as in Theorem \ref{th:univ}.
\bibliographystyle{amsplain}
\bibliography{heavygirko}

\section*{Acknowledgment}

The author is indebted to Tim Rogers for pointing reference \cite{R} which has initiated this work, and thanks Djalil Chafa\"i and Manjunath Krishnapur for sharing their enthusiasm on non-hermitian random matrices.

\end{document}